%% file: root.tex
\documentclass[letterpaper, 10 pt, journal, twoside]{IEEEtran}

\IEEEoverridecommandlockouts                              

\usepackage{graphics} 
\usepackage{epsfig} 
\usepackage{epstopdf}
\usepackage{amsmath} 
\usepackage{amssymb}  
\usepackage{amsthm}
\usepackage{amsfonts}
\usepackage{subfigure}
\input{ACSD-preamble}

\title{
Adaptive dynamic programming using Lyapunov function constraints
}

\author{Thomas G{\"o}hrt and Pavel Osinenko and Stefan Streif
\thanks{The authors are with the Automatic Control and System Dynamics Lab at the Technische Universit{\"a}t Chemnitz, Germany
        {\tt\small \{thomas.goehrt,pavel.osinenko,stefan.streif\}
@etit.tu-chemnitz.de}}%
}

\begin{document}

\maketitle
\thispagestyle{empty}
\pagestyle{empty}

\begin{abstract}
This work is concerned with a stabilizing adaptive dynamic programming (ADP) approach to approximate solution of a given infinite-horizon optimal control problem.
Since the latter problem cannot, in general, be solved exactly, a parametrized function approximator for the infinite-horizon cost function is introduced in ADP (so called ``critic'').
This critic is used to adapt the parameters of the function approximator.
The so called ``actor'' in turn derives the optimal input of the system.
It is a notoriously hard problem to guarantee closed-loop stability of ADP due to the use of approximation structures in the control scheme.
Since at least stabilizability is always assumed in the analyses of ADP, it is justified to invoke a respective Lyapunov function.
The proposed ADP scheme explicitly uses the said Lyapunov function to simultaneously optimize the critic and guarantee closed-loop stability.
A Hessian-free optimization routine is utilized for the actor and critic optimization problems.
Convergence to prescribed vicinities of the optima is shown.
A computational study showed significant performance improvement for the critic-based approach compared a nominal stabilizing controller for a range of initial conditions.
\end{abstract}
\section{Introduction}\label{sec:intro}

\IEEEPARstart{C}{onsider} a control-affine nonlinear system
\begin{equation}\label{eq:system}
\dot x(t) = f(x(t)) + g(x(t)) u,
\end{equation}
where $x \in \R^n$ denotes the state, $u \in \R^m$ denotes the input, $f : \R^n \to \R^n$ denotes the internal dynamics model and $g : \R^n \to \R^{n  \times m}$ denotes the input coupling function.

The infinite-horizon optimal control problem can be stated as follows:
\begin{equation}
\label{eq:HJB-ctrl}
	H^*(x_0) = \min\limits_v \int \limits_0^{\infty} r(x(t), v(x(t))) \mathrm{dt}, \quad x_v(0) = x_0, 
\end{equation}
where $r : \R^n \times \R^m \to \R_{\geq0}$ denotes the utility function and $v : \R^n \to \R^m$ denotes the so called policy. 
The so called \emph{optimal value function} $H^*$ was shown to satisfy the respective Hamilton-Jacobi-Bellman (HJB) equation \cite{Kirk2004-opt-ctrl}.
The HJB is a partial differential equation and is notoriously hard to solve analytically. 
In the framework of dynamic programming (DP) \cite{Bellman1957-DP}, the problem is addressed by state space discretization and propagation of the value function backwards in time. 
For online control, such an approach is oftentimes not suitable, which motivated the development of adaptive dynamic programming (ADP). 
ADP uses a parametrized approximations of the value function (see \cite{Sutton1998-RL,Lewis1995-opt-ctrl}) and develops parameter adjustment laws. 
The respective approximator is referred to as the so called \emph{critic}. 
Specific formulations are based \eg on Galerkin approximation \cite{Beard1997-ADP-Galerkin}, neural networks \cite{Werbos1990-menu,Bertsekas1995-NDP}, polynomial models \cite{Jiang2015-ADP-poly,Kamalapurkar2016-ADP-RL-regul}, etc.
It should be remarked that there is in general no quantifiable relation between the value function and its approximation.
Some approaches to ADP rely on strong assumptions on the value function itself \cite{Jiang2015-ADP-poly}.
Once an approximation is computed, the optimal controller is found by optimizing the HJB in terms of the critic.
This is referred to as the so called \emph{actor}.

One of the major challenges of ADP is guaranteeing closed-loop stability \cite{Balakrishnan2008-ADP-issues}. 
Existing analyses often assume that the optimization problem (\ref{eq:HJB-ctrl}) be analytically solvable at all times \cite{Al-Tamimi2008-VI-ADP,Liu2014-PI-ADP-conv-prf,Heydari2014-better-prf-ADP}.
A relaxed version of an infinite horizon optimal control problem which preserves stability was addressed in \cite{Jiang2015-ADP-poly} by assuming a Lyapunov function to be sum of squares and using semidefinite programming.
In general, stabilizing controllers can be made optimal only under conformity of the value function and the respective control Lyapunov function \cite{Primbs1999-opt-ctrl-CLF}.
In \cite{Primbs1999-opt-ctrl-CLF}, a particular example of a two-dimensional nonlinear harmonic oscillator is considered and a comparison of a stabilizing and optimal controller along with the respective value functions is provided.
In \cite{wang2018intelligent}, it is assumed that a positive definite function serves as a Lyapunov function for the approximate optimal controller. 
This function is incorporated into the adaptation scheme to guarantee stability.
A further observation can be made from the literature that ADP often relies on exhaustive offline learning phases for the critic and obtaining a model of the system dynamics in the form of \eg neural networks \cite{wang2012optimal}.

In this work, it is suggested to consider a pure online ADP scheme under a stabilizability constraint provided by a Lyapunov function, not necessarily related to an approximate optimal controller.
Such an approach relaxes the original problem \eqref{eq:HJB-ctrl}.
In general, as pointed out above, no relation between a relaxed and the original formulation of \eqref{eq:HJB-ctrl} is known without strong assumptions on the value function.
For instance, in the linear case, it is at least known that the value function is quadratic.
In such a case, a suitable Lyapunov function might be also chosen quadratic.
The resulting optimization problems in this work are shown to be time-varying and strongly convex.
Time-varying optimization problems have, in general, time-varying optima \cite{2016-fazlyap-ip}.
Recent developments in optimization algorithms for continuous time systems by means of ordinary differential equations show that the inverse Hessian of the cost function is required to converge to and track the optimum \cite{2016-fazlyap-ip}.
In particular, the optimization routine consists of a tracking controller, which compensates for the time-varying behavior of the optimization problem, and an update rule for convergence to the optimum.
However, calculating the inverse Hessian might be expensive for practical purposes \cite{2017-simonetto-first-order-pred-corr,2017-simonetto-pred-corr}.

The main contribution of this work is an ADP approach with guaranteed closed loop stability using a Lyapunov function constraint. 
Gradient descent optimization routines with time-varying gains are suggested to avoid the costly calculation of inverse Hessian matrices. 
It is shown that convergence to prescribed vicinites can be achieved by specifically tuned gains for gradient descent.
The suggested approach comprises of a pure online learning of the critic part of the algorithm.
A cost improvement of the controller using the proposed approach compared to a nominal controller is demonstrated in a case study.

The rest of the paper is organized as follows.
The next section introduces required preliminaries.
In Section \ref{sec:prob-statement}, the considered problem statement is given.
The update rules  as well as an analysis of optimization problems are provided in Section \ref{sec:analysis}. 
Finally, a case study and discussion are given in Section \ref{sec:results}.

\section{Preliminaries}\label{sec:preliminaries}
This section introduces required preliminaries and notation.
Function arguments are omitted from here on to simplify presentation and are only provided at the first instance or when necessary.

In the following, $\Vert \cdot \Vert$ denotes the Euclidean 2-norm for vectors and the spectral norm for matrices.
Further, $\R^+$ denotes the positive reals and $\R_0^+$ denotes the positive reals including zero, respectively, $\R^-$ represents the negative reals.
The notation $f_x$ for a function $f$ means the partial derivative of $f$ with respect to $x$.

The central assumption in this work is:
\begin{assumption}\label{ass:lyap}
There exists a continuously differentiable feedback controller $v(x)$ and a twice continuously differentiable radially unbounded Lyapunov function $V$ such that for all $x \in \R^n \setminus \{0\}$, it holds that
\begin{equation}\label{eq:lyap_equation}
\dot V(x) = V_x^\top(x) ( f + g v(x)) \leq - k(x),
\end{equation}
where $k(x)>0$ for all $x \in \R^n \setminus \{0\}$ and $k(0)=0$.
Moreover it is assumed that $k$ is continuously differentiable.
\end{assumption}
\begin{remark}
	Assumption \ref{ass:lyap} in this work can be relaxed to the case of a local Lyapunov function.
	The global case is used rather for simplicity of notation.
\end{remark}
The convergence analysis in Section \ref{sec:conv-anal} will require the notion of persistence of excitation:
\begin{definition}[Persistence of excitation \cite{Modares2013-PI-input-constr}]\label{def:pe}
A continuous matrix-valued function $A$ of time $t$ is called uniformly $T$-persistently excited at a \emph{level} $\delta$ if, for all $t$ the matrix
\begin{equation}
\int\limits_{t-T}^t A(\tau) \mathrm{d \tau} - \delta I
\end{equation}
is positive semidefinite. Here, $I$ denotes the identity matrix.
\end{definition}
\begin{definition}[Strongly convex]\label{def:strong-convex}
A real-valued twice continuously differentiable function $J : \mathcal{D} \to \R$ is called strongly convex with parameter $m$, if $J_{xx}(x) - m I$ is positive semidefinite for all $x \in \mathcal{D}$, where $J_{xx}$ denotes the Hessian of $J$.
\end{definition}
The utility function is assumed 
\begin{equation}\label{eq:utility_fcn}
r(x,u)\coloneqq q(x) + p(u),
\end{equation}
where $q(0)=0$ and $q(x) > 0$ for all $x \in \R^n \setminus \{0 \}$. 
The function $p : \R^m \to \R_0^+$ is assumed to be twice continuously differentiable and strongly convex with parameter $m_1$.

The critic is chosen in a parametric form as follows:
\begin{equation}\label{eq:critic}
\hat J(x,w(t)) \coloneqq w^\top(t) \varphi(x),
\end{equation}
where $\varphi : \R^n \to \R^d, \varphi \in \mathcal{C}^1 $ is the regressor vector and $w \in \R^d$ is the parameter vector.
\begin{definition}[Barrier function]\label{def:barrier-fcn}
A twice continuously differentiable convex function $B : \R_0^- \to \R \cap \{\infty\}$ is called a \emph{barrier function}, if $B(z) < \infty$ for all $z <0$ and $\lim\limits_{z \to 0^-} B(z) \to \infty$.
\end{definition}

Finally, the internal dynamics model and the input coupling function are assumed to be continuously differentiable.

The next section poses the problem statement.

\section{Problem statement}\label{sec:prob-statement}

Consider the following optimization problem for the actor:
\begin{equation}
\begin{split}
\min_{u} \quad & \Lambda (x,u,w) \\
\mathrm{s.t.} \quad & \dot x = f + g u, \\
& z(x,u) \leq \gamma, \\
\end{split}
\end{equation}
with 
\begin{equation}\label{eq:def-Ju}
\Lambda(x,u,w) \coloneqq \underbrace{w^\top \dot \varphi}_{\hat J_x^\top \dot x} + r
\end{equation}  
and 
\begin{equation}\label{eq:lyap-constraint}
z(x,u)\coloneqq V_x^\top ( f + g u) + k,
\end{equation}
where $0 <\gamma$.
The function $z$ is defined based on Assumption \ref{ass:lyap}.
The cost function $\Lambda$ is known as the Bellman error in ADP and is an infinitesimal version of \eqref{eq:HJB-ctrl} using the approximator $\hat J$.
The choice of this cost function is rather standard in ADP \cite{Lewis2009-ADP}.
This optimization problem is relaxed using a barrier function $B$ as follows
\begin{equation}\label{eq:actor-opt-problem}
\begin{split}
\min_{u} \quad & (\Lambda  + \mu(t) B(z - \gamma))\\
\mathrm{s.t.} \quad & \dot x = f + g u, \\
\end{split}
\end{equation}
where $\mu$ satisfies $\mu(t) > 0$ and $\dot\mu(t)<0$ for all $t \geq 0$.
Since $z$ converges to zero, if $x$ converges to zero, $\gamma >0$ ensures a well defined optimization problem even if $z=0$.
Notice that a solution of \eqref{eq:actor-opt-problem} is in general merely a suboptimal solution for the optimization problem \eqref{eq:HJB-ctrl} (please refer also to the discussion in the introduction).

Further, consider the following optimization problem for the critic:
\begin{equation}\label{eq:critic-opt-problem}
\begin{split}
\min_{w} \quad & J (x_T,u_T,w) \\
\mathrm{s.t.} \quad & \dot x = f + g u, 
\end{split}
\end{equation}
where 
\begin{equation}\label{eq:def-Jw}
J (x_T,u_T,w) \coloneqq \int\limits_{t-T}^t \Lambda^2 (x(\tau),u(\tau),w(t)) \mathrm{d \tau}.
\end{equation} 
The right-hand side of \eqref{eq:def-Jw} is known as the interval reinforcement form \cite{Lewis2009-ADP} and is the integral over the squared Bellman error.
The notation $x_T$ indicates that $x$ is a vector-valued function over the moving time interval $[t-T, t]$ \ie, constant $w$ is sought over the time interval $[t-T, t]$.

The considered problem within this work is now stated as follows.
Find update rules for the optimization problem \eqref{eq:actor-opt-problem},
\begin{equation}\label{eq:update-actor-proto}
\dot u = \vartheta_u(x,u,w), \quad u(0) \in \mathcal{U}
\end{equation}
and for the optimization problem \eqref{eq:critic-opt-problem},
\begin{equation}\label{eq:update-critic-proto}
\dot w = \vartheta_w (x_T,u_T,w)
\end{equation}
such that $\lim\limits_{t \to \infty} \Vert u(t) - u^*(t) \Vert \leq \varepsilon_1$ and $\lim\limits_{t \to \infty} \Vert w(t) - w^*(t) \Vert \leq \varepsilon_2$, where
\begin{equation}\label{eq:actor-opt-problem-arg}
\begin{split}
u^*(t) \coloneqq \argmin_{u} \quad & (\Lambda  + \mu B)\\
\mathrm{s.t.} \quad & \dot x = f + g u \\
\end{split}
\end{equation}
and
\begin{equation}\label{eq:critic-opt-problem-arg}
\begin{split}
w^*(t)\coloneqq \argmin_{w} \quad & J  \\
\mathrm{s.t.} \quad & \dot x = f + g u.
\end{split}
\end{equation}

The next section introduces the update rules for the proposed optimization problems.
Further, these optimization problems as well as the update rules will be analyzed with respect to convexity, stability of the closed-loop system, feasibility as well as convergence of the optimization variables to vicinities of the optimal parameters.


\section{Update rules and analysis}\label{sec:analysis}

This section is concerned with the formulation and analysis of the update rules as well as the analysis of the proposed optimization problems.

\subsection{Convexity analysis}\label{sec:convex}

It is started with the convexity analysis of the critic optimization problem \eqref{eq:critic-opt-problem}.

\begin{assumption}\label{ass:critic-pe}
The matrix $\dot \varphi(x) \dot \varphi^\top(x)$ is $T$-persistently excited at level $\delta_1$.
\end{assumption}
\begin{lemma}
Let Assumption \ref{ass:critic-pe} hold.
The cost function $J$ of the optimization problem \eqref{eq:critic-opt-problem} is strongly convex with parameter $\delta_1$.
\end{lemma}
\proof
Using \eqref{eq:utility_fcn} and \eqref{eq:def-Ju}, the cost function $J$ from \eqref{eq:def-Jw} can be expressed as
\begin{equation}\label{eq:def-Jw-QP}
J = w^\top H(x_T,u_T) w + w^\top \rho(x_T, u_T) + \nu(x_T, u_T),
\end{equation}
where
\begin{align}
H(x_T,u_T) &=  \int\limits_{t-T}^t  \dot \varphi(x) \dot \varphi^\top(x) \mathrm{d\tau},\\
\rho(x_T, u_T) &= \int\limits_{t-T}^t  2 \dot \varphi(x) r(x,u) \mathrm{d\tau} ,\\
\nu(x_T, u_T) &=  \int\limits_{t-T}^t  r^2(x,u)\mathrm{d\tau}.
\end{align}

The Hessian of the cost function \eqref{eq:def-Jw-QP} is simply $H$ \ie, under assumption \ref{ass:critic-pe}, the cost function \eqref{eq:def-Jw-QP} is strongly convex with parameter $\delta_1$.
\endproof
The critic optimization problem is unconstrained since the dynamic can be put directly into $J$, hence feasibility is not an issue.

\begin{lemma}
	\label{lem:actor-convexity}
Consider the system \eqref{eq:system}, the utility function \eqref{eq:utility_fcn} and the optimization problem \eqref{eq:actor-opt-problem}.
Let Assumption \ref{ass:lyap} hold.
The cost function $\bar \Lambda$ defined in \eqref{eq:def-actor-cost-barrier} is strongly convex with parameter $m_1$.
\end{lemma}
\proof
First, observe that the inequality constraint $z \leq \gamma$ from \eqref{eq:lyap-constraint} is affine in $u$.
The cost function $\Lambda$ is strongly convex in $u$ since $p$ is strongly convex in $u$ with parameter $m_1$.
The barrier function $B$ is convex by definition and has a affine function as argument, hence $B(z-\gamma)$ is convex in $u$ for all $x \in \R^n$.
Further, the sum of a convex and strongly convex function is again strongly convex \ie, the cost function $\bar \Lambda$ of the considered optimization problem \eqref{eq:actor-opt-problem}  is strongly convex with parameter $m_1$.
\endproof

\subsection{Stability analysis}
\label{sec:Stab-anal}
\begin{theorem}
Consider the system \eqref{eq:system} and let the initial condition be $x(0)=x_0$.
Let Assumption \ref{ass:lyap} hold.
For any continuously differentiable function $u: \R_0^+ \to \mathcal{U}$ that satisfies
\begin{equation}\label{eq:lyap-lem}
\dot V = V_x^\top ( f + g u) \leq - k + \gamma, \quad \forall t \in \R_0^+, x(0)=x_0,
\end{equation}
along the trajectory, the state $x$ converges to the set $\mathcal{G} \coloneqq \{x \in \R^n : k(x) \leq \gamma\}$.
\end{theorem}
\proof
	First, notice that $k$ is a positive definite function, hence the set $\mathcal{G} \subset \R^n$ is compact.
	Define, further the set $E\coloneqq \{x \in \R^n : \dot V =0\}$.
	Since for all $x \in \R^n \setminus \mathcal{G}$, one has $\dot V < 0$, it follows that $E \subseteq \mathcal{G}$.
	By Lasalles invariance priciple \cite[Theorem 4.4]{khalil2002-nonlinear} $x$ converges to the largest invariant set in $E$, hence it converges to the set $\mathcal{G}$.
\endproof

\subsection{Update rules}\label{sec:update-rules}

As already mentioned in Section \ref{sec:intro}, asymptotic convergence to the optimal parameters of a time-varying strongly convex optimization problem, in general, requires a tracking controller for the internal dynamic.
Such a tracking controller in turn may need to use the inverse Hessian of the cost function, which might be expensive in practice \cite{2016-fazlyap-ip}.
Here, it is suggested to use the gradient descent with specially tuned gains and without using the inverse Hessian.
It should be pointed out that various other optimization routines are possible, in particular from the regret analysis.
This section considers a particular one, which is sufficient for the purposes of the current study.

First, define
\begin{equation}\label{eq:def-actor-cost-barrier}
\bar \Lambda(x,u,w,t) \coloneqq \Lambda  + \mu B.
\end{equation}
The update rule \eqref{eq:update-actor-proto} for the actor optimization problem \eqref{eq:actor-opt-problem} is now defined as
\begin{equation}\label{eq:update-rule-actor}
\dot u \coloneqq - \alpha(t)  \bar \Lambda_u, \quad u(0)= v(x(0)).
\end{equation}
The update rule \eqref{eq:update-critic-proto} for the critic optimization problem \eqref{eq:critic-opt-problem} is defined as
\begin{equation}\label{eq:update-rule-critic}
\dot w \coloneqq - \beta(t) J_w.
\end{equation}

\subsection{Feasibility analysis}

In this section, feasibility of the optimization problem \eqref{eq:actor-opt-problem} is analyzed.
It is shown that $u$ remains feasible along the trajectory of the system \eqref{eq:system}.
Consider the following
\begin{lemma}\label{lem:recursive-feasibility}
Consider the system \eqref{eq:system}, the optimization problem \eqref{eq:actor-opt-problem} and the update rule \eqref{eq:update-rule-actor}.
Let Assumption \ref{ass:lyap} hold.
For every initial condition $x \in \R^n$, the solution $u(t)$ to \eqref{eq:update-rule-actor} satisfies $u(t) \in \mathcal{H}(x(t))$ for all $ t\geq 0$, if $u(0) \in \mathcal{H}(x(0))$, where $\mathcal{H}(x)\coloneqq\{u \in \R^m : z(x,u)\leq \gamma\}$.
\end{lemma}
\proof
First fix some $x \in \R^n$.
Notice that under Assumption \ref{ass:lyap} there exists a feasible $u$ for all $x \in \R^n$.
Since $u$ starts feasible, it has to be shown that $u$ does not cross the hyperplane defined by $z(x,u)=\gamma$ from the feasible to the infeasible area.
Define the set of boundary points as $\partial\mathcal{H}(x) \coloneqq \{u \in \R^m : z(x,u) = \gamma \}$.
It is shown that if $u$ approaches $u_\partial$ where $u_\partial \in \partial\mathcal{H}(x)$, $\dot u$ is such that $u$ moves away from the hyperplane defined by $\partial\mathcal{H}(x)$.
This can be formalized as follows:
\begin{equation}
	\forall u_{\partial} \in \partial \mathcal H(x) \lim_{u \to u_{\partial}} - \frac{B_{u}^\top \dot u}{B_{u}^\top B_u}  > 0.
\end{equation}
A direct calculation yields
\begin{equation}
\begin{split}
&\lim_{u \to u_{\partial}} - \frac{B_{u}^\top \dot u}{B_{u}^\top B_u} =\alpha \lim_{u \to u_{\partial}} \left( \frac{B_{u}^\top \Lambda_u}{B_{u}^\top B_u} + \mu \underbrace{\frac{B_{u}^\top B_{u}}{B_{u}^\top B_u}}_{=1} \right) \\
&=\alpha \left(\mu + \lim_{u \to u_{\partial}}  \frac{B_z z_{u}^\top \Lambda_u}{ B_z^2 z_{u}^\top z_u} \right) = \alpha \left(\mu + \lim_{u \to u_{\partial}}  \frac{z_{u}^\top \Lambda_u}{ B_z z_{u}^\top z_u} \right).
\end{split}
\end{equation}
Since $z$ is linear in $u$, $z_u$ is a constant vector.
The cost function $\Lambda$ is globally defined and twice continuously differentiable, hence
\begin{equation}\label{eq:limit-cost-fcn}
\lim\limits_{u \to u_\partial} \Lambda_u = \Lambda_u (x,u_\partial , w) < \infty.
\end{equation}
Furthermore, the gradient of the barrier function $B$ tends to infinity at the boundary \ie
\begin{equation}\label{eq:limit-barrier-gradient}
\lim\limits_{u \to u_\partial} B_z \to \infty.
\end{equation}
From \eqref{eq:limit-cost-fcn} and \eqref{eq:limit-barrier-gradient}, it follows that
\begin{equation}
\alpha \left(\mu + \underbrace{\lim_{u \to u_{\partial}} \frac{z_{u}^\top \Lambda_u}{ B_z z_{u}^\top z_u}}_{=0} \right)= \alpha \mu >0,
\end{equation}
which completes the proof.
\endproof
\begin{remark}
Lemma \ref{lem:recursive-feasibility} can be interpreted as a kind of continuous time equivalence of recursive feasibility known from discrete time optimal control problems \cite{lofberg2012oops}.
\end{remark}

The next subsection analyzes the convergence to prescribed vicinities of the optima for the optimization problems \eqref{eq:actor-opt-problem} and \eqref{eq:critic-opt-problem}.

\subsection{Convergence analysis}\label{sec:conv-anal}

The following theorem shows that under properly chosen gains of the update rules, the actor and critic parameters converge to prescribed vicinities of the respective optima.

\begin{theorem}\label{thm:convergence}
Consider the system \eqref{eq:system}, the optimization problem \eqref{eq:actor-opt-problem} together with the update rule \eqref{eq:update-rule-actor} and the optimization problem \eqref{eq:critic-opt-problem} together with the update rule \eqref{eq:update-rule-critic}.
Let Assumption \ref{ass:lyap} and \ref{ass:critic-pe} hold.
There exists gains $\alpha$ and $\beta$ such that $\lim\limits_{t \to \infty} \Vert  u(t) -u^*(t) \Vert \leq \varepsilon_1$ and $\lim\limits_{t \to \infty} \Vert  w(t)-w^*(t) \Vert \leq \varepsilon_2$ holds.
\end{theorem}
\proof
First, define $\tilde u \coloneqq u -u^*$ and $\tilde w \coloneqq w -w^*$.
Consider the following positive definite function 
\begin{equation}
V(\Lambda_u, J_w)=\underbrace{\frac{1}{2} \bar \Lambda_u^\top \bar \Lambda_u }_{=: V_1(\Lambda_u)} + \underbrace{\frac{1}{2}J_w^\top J_w.}_{=:V_2(J_w)}
\end{equation}
The derivative reads as
\begin{equation}
\dot V = \underbrace{\bar \Lambda_u^\top  \dot{\bar{\Lambda}}_u}_{=\dot V_1} + \underbrace{J_w^\top \dot{J}_w}_{=\dot V_2}.
\end{equation}
First, the term $\dot V_1$ is addressed
\begin{equation}
\dot V_1 = \bar\Lambda_u^\top (\bar\Lambda_{uu}  \dot u + \bar\Lambda_{ux}  \dot x + \bar\Lambda_{uw} \dot w + \bar\Lambda_{ut}  ).
\end{equation}
Using \eqref{eq:update-rule-actor}, one can further write
\begin{equation}\label{eq:actor-lyap-bound}
\begin{split}
\dot V_1  = - \alpha \bar\Lambda_u^\top \bar\Lambda_{uu}\bar\Lambda_u + \bar\Lambda_u^\top ( \bar\Lambda_{ux}  \dot x + \bar\Lambda_{uw} \dot w + \bar\Lambda_{ut} ). \\
\end{split}
\end{equation}
Since $\bar\Lambda$ is a strongly convex function with parameter $m_1$ (see Section \ref{sec:convex}), one has
\begin{equation} \label{eq:bound-Hessian}
x^\top \Lambda_{uu} x \geq m_1 x^\top x.
\end{equation}
Further, for every strongly convex function holds that \cite{nesterov2013introductory}
\begin{equation}
\bar\Lambda^* \geq \bar\Lambda + \bar\Lambda_u^\top \underbrace{(u^* -u)}_{=- \tilde u} + \frac{m_1}{2} \Vert \tilde u \Vert^2.
\end{equation}
Rearranging yields
\begin{equation}
\bar\Lambda_u^\top \tilde u \geq \underbrace{\bar\Lambda - \bar\Lambda^*}_{>0} + \frac{m_1}{2} \Vert \tilde u \Vert^2 \geq \frac{m_1}{2} \Vert \tilde u \Vert^2 \geq 0.
\end{equation}
Now, the following chain of inequalities can be written
\begin{equation}
\frac{m_1^2}{4} \Vert \tilde u \Vert^4 \leq (\bar\Lambda_u^\top \tilde u)^2 \leq \Vert \bar\Lambda_u \Vert^2 \Vert \tilde u \Vert^2
\end{equation}
and further
\begin{equation}
\frac{m_1^2}{4} \Vert \tilde u \Vert^2 \leq  \Vert \bar\Lambda_u \Vert^2 = \bar\Lambda_u^\top \bar\Lambda_u.
\end{equation}
Multiplying by $-1$ yields an upper bound on $\bar\Lambda_u^\top \bar\Lambda_u$,
\begin{equation}
-\bar\Lambda_u^\top \bar\Lambda_u \leq -\frac{m_1^2}{4} \Vert \tilde u \Vert^2.
\end{equation}
Using this upper bound together with \eqref{eq:bound-Hessian}, one can further bound \eqref{eq:actor-lyap-bound}
\begin{equation}
\begin{split}
\dot V_1  \leq - \alpha\frac{m_1^3}{4} \Vert \tilde u \Vert^2 + \bar\Lambda_u^\top ( \bar\Lambda_{ux}  \dot x + \bar\Lambda_{uw} \dot w + \bar\Lambda_{ut}  ).\\
\end{split}
\end{equation}
Choose $\alpha$ as
\begin{equation}
\label{eq:definition-alpha}
\alpha= \frac{4}{m_1^3 \varepsilon_1^2} (\bar\Lambda_u^\top ( \bar\Lambda_{ux}  \dot x + \bar\Lambda_{uw} \dot w + \bar\Lambda_{ut}  ))^2
\end{equation}
which yields
\begin{equation}\label{eq:lyap-actor-decay}
\begin{split}
\dot V_1  \leq& \left(1 - \frac{\Vert \tilde u \Vert^2}{\varepsilon_1^2}  \bar\Lambda_u^\top ( \bar\Lambda_{ux}  \dot x + \bar\Lambda_{uw} \dot w + \bar\Lambda_{ut} ) \right)\\
&\bar\Lambda_u^\top ( \bar\Lambda_{ux}  \dot x + \bar\Lambda_{uw} \dot w + \bar\Lambda_{ut}).
\end{split}
\end{equation}
Independent of the sign of $\bar\Lambda_u^\top ( \bar\Lambda_{ux}  \dot x + \bar\Lambda_{uw} \dot w + \bar\Lambda_{ut})$, if $\Vert \tilde u \Vert$ is greater than $\varepsilon_1$ then $\dot V_1 < 0$.

Analogously, $\dot{V}_2$ can be written as
\begin{equation}
\begin{split}
\dot{V}_2 &= - 2 \beta J_w^\top H J_w + J_w^\top ( 2 \dot H w + \dot \rho) \\
& \leq - \beta \frac{\delta_1^3}{2} \Vert \tilde w \Vert^2 + J_w^\top ( 2 \dot H w + \dot \rho),
\end{split}
\end{equation}
where the last inequality uses Assumption \ref{ass:critic-pe}.
Choosing $\beta$ as
\begin{equation}
\label{eq:definition-beta}
\beta = \frac{2}{\delta_1^3 \varepsilon_2^2} (J_w^\top ( 2 \dot H w + \dot \rho))^2
\end{equation}
yields
\begin{equation}
\begin{split}
\dot{V}_2 \leq&  \left( 1 - \frac{\Vert \tilde w \Vert^2}{\varepsilon_2^2} J_w^\top ( 2 \dot H w + \dot \rho )\right)  J_w^\top ( 2 \dot H w + \dot \rho ).
\end{split}
\end{equation}
Again, independent of the sign of $J_w^\top ( 2 \dot H w + \dot \rho)$, if $\Vert \tilde w \Vert$ is greater than $\varepsilon_2$ then $\dot V_2 < 0$.

Now, define $\tilde \Lambda_u \coloneqq \Lambda_u - \Lambda_u^*$ and $\tilde J_w = J_w -J_w^*$, where $\Lambda_u^* = \Lambda_u(x,u^*,w,t)$ and, respectively, $J_w^* \coloneqq J_w(x,u,w^*)$.
Now, introduce the following sets
\begin{align}
\mathcal{D} \coloneqq& \{\Lambda_u \in \R^m, J_w \in \R^d : \notag\\
 &V(\tilde \Lambda_u, \tilde J_w) \leq V(\tilde \Lambda_{u_0}, \tilde J_{w_0})\}, \\
\mathcal{B} \coloneqq& \{ (\tilde \Lambda_u, \tilde J_w) \in \mathcal{D} :  \Vert \tilde u \Vert \leq \varepsilon_1 \land \Vert \tilde w \Vert \leq \varepsilon_2 \}, \\
\partial\mathcal{B} \coloneqq& \{ (\tilde \Lambda_u, \tilde J_w) \in \mathcal{D} :  \Vert \tilde u \Vert = \varepsilon_1 \land \Vert \tilde w \Vert = \varepsilon_2 \}, \\
\Omega \coloneqq& \{(\tilde \Lambda_u, \tilde J_w) \in \mathcal{D} : \dot V \leq 0\}.
\end{align}
Define the following ``stripes'' in $\mathcal{D}$:
\begin{align}
\mathcal{S}_u &\coloneqq  \{(\tilde \Lambda_u, \tilde J_w) \in \mathcal{D} : \Vert \tilde u \Vert \leq \varepsilon_1 \land \Vert \tilde w \Vert \geq \varepsilon_2\}, \\
\mathcal{S}_w &\coloneqq  \{(\tilde \Lambda_u, \tilde J_w) \in \mathcal{D} : \Vert \tilde u \Vert \geq \varepsilon_1 \land \Vert \tilde w \Vert \leq \varepsilon_2\}.
\end{align}
Let $\mathcal{S} \coloneqq \mathcal{B} \cup \mathcal{S}_u \cup \mathcal{S}_w$.
Further split the set $\Omega$ into
\begin{align}
\Omega_u &\coloneqq \{(\tilde \Lambda_u, \tilde J_w) \in (\mathcal{S}_u \cup \mathcal{B}) : \dot V \leq 0\}, \\
\Omega_w &\coloneqq \{(\tilde \Lambda_u, \tilde J_w) \in (\mathcal{S}_w \cup \mathcal{B}) : \dot V \leq 0\}.
\end{align}
Finally, define sets for which $\dot V, \dot V_1$ and $\dot V_2$ are equal zero:
\begin{align}
\mathcal{E} &\coloneqq \{(\tilde \Lambda_u, \tilde J_w) \in \Omega : \dot V=0\}, \\
\mathcal{E}_u &\coloneqq \{(\tilde \Lambda_u, \tilde J_w) \in \Omega_u: \dot V_1=0\}, \\
\mathcal{E}_w &\coloneqq \{(\tilde \Lambda_u, \tilde J_w) \in \Omega : \dot V_2=0\}. 
\end{align}
If $(\tilde \Lambda_u, \tilde J_w) \in \Omega$, then, by the Lyapunov theory \cite{khalil2002-nonlinear}, it follows that $\tilde \Lambda_u, \tilde J_w$ converges to the set $\mathcal{E}$.
Since $\dot V <0$ for all $(\tilde \Lambda_u, \tilde J_w) \in \Omega\setminus\mathcal{S}$ one has $\mathcal{E} \subseteq \mathcal{S}$.
Now, two different cases need to be considered.
In the first case, consider that $(\tilde \Lambda_u, \tilde J_w) \in \Omega_u$.
Based on $\dot V_1<0$, $\tilde \Lambda_u$ converges to the set $\mathcal{E}_u$ with $\mathcal{E}_u \subseteq \Omega_u \subseteq (\mathcal{S}_u \cup \mathcal{B})$.
In the second case, consider $(\tilde \Lambda_u, \tilde J_w) \in \Omega_w$.
Based on $\dot V_2<0$, $\tilde J_w$ converges to the set $\mathcal{E}_w$ with $\mathcal{E}_w \subseteq \Omega_w \subseteq (\mathcal{S}_w \cup \mathcal{B})$.

The resulting set $\mathcal{E}_B$ is the intersection between those sets defined by $\mathcal{E}_B \subseteq (\mathcal{E} \cap \mathcal{E}_u \cap \mathcal{E}_w)$.
Using the fact that $\mathcal{S}_u \cap \mathcal{S}_w = \partial\mathcal{B}$ and $\partial\mathcal{B} \cup \mathcal{B}=\mathcal{B}$ , one obtains the following relation
\begin{equation}
\begin{split}
\mathcal{E}_B &\subseteq (\mathcal{E} \cap \mathcal{E}_u \cap \mathcal{E}_w) \subseteq (\mathcal{S} \cap (\mathcal{S}_u \cup \mathcal{B}) \cap (\mathcal{S}_w \cup \mathcal{B})) \\
&\subseteq ((\mathcal{B} \cup \mathcal{S}_u \cup \mathcal{S}_w) \cap (\mathcal{S}_u \cup \mathcal{B}) \cap (\mathcal{S}_w \cup \mathcal{B})) \subseteq \mathcal{B}.
\end{split}
\end{equation}
It follows that $(\tilde \Lambda_u, \tilde J_w)$ converge to prescribed vicinities $\mathcal{B}$, which completes the proof.
\endproof

Notice that all constants and functions are known in \eqref{eq:definition-alpha} and \eqref{eq:definition-beta} \ie $\alpha$ and $\beta$ can be calculated online.
It can be further seen that the choice of the vicinities by means of $\varepsilon_1, \varepsilon_2$ influences the gains $\alpha$ and $\beta$. 
Smaller values of $\varepsilon_1, \varepsilon_2$ lead to the necessity of increasing of $\alpha$ and $\beta$.

The next section presents a case study.
\section{Case Study and Discussion}\label{sec:results}

Consider the following dynamic system
\begin{equation}
\dot x =
\begin{pmatrix}
-x_1 - x_2^2 \\
x_1 x_2
\end{pmatrix}
+ 
\begin{pmatrix}
0 \\
1
\end{pmatrix}
u.
\end{equation}
which is constructed using converse optimality (see \cite{Doyle1996-converse-HJB,nevistic1996-constrained-converse,primbs1996-converse}).
The utility function is assumed $r = x^\top x + 0.25 u^2$.
The system is asymptotically stable under the controller $u \coloneqq -x_1 - x_2$ as shown by the Lyapunov theory.
Specifically, consider the positive definite function $V(x)= 0.5 x^\top x$
Its derivative reads as
\begin{equation}
\dot V = x^\top \dot x = -x_1^2 - x_1 x_2 - x_2^2 = - x^\top 
\underbrace{\begin{pmatrix}
1 & 0.5\\
0.5 & 1
\end{pmatrix}}_{=:P} x.
\end{equation}
The matrix $P$ is positive definite, hence $\dot V <0$ for all $x \in \R^n\setminus{0}$, which shows asymptotic stability.
Now, define $k=0.1 \Vert x \Vert^2$ and let $\gamma = 0.01$.
The barrier function $B$ is chosen as $B = - z^{-1}$.
The prescribed vicinities are chosen as $\varepsilon_1 = 0.5$ and $\varepsilon_2 = 0.5$.
The evolution of $\mu$ is chosen as $\dot \mu = - \mu$ with $\mu(0)=1$.
The window for the cost function $J$ (see \eqref{eq:def-Jw}) is chosen as $T=0.1$.

Within this case study, the controller resulting from the proposed approach \ie \eqref{eq:update-rule-actor} and \eqref{eq:update-rule-critic} together with the gains derived in Theorem \ref{thm:convergence}, is compared to the nominal controller from Assumption \ref{ass:lyap} and the optimal controller known from the converse approach.
The initial conditions for the controller using the proposed approach is $v(x(0))$.

The performance mark for the controllers using \eqref{eq:utility_fcn} is considered
\begin{equation}\label{eq:performance-mark}
C=\int\limits_0^{T_{\text{sim}}} r(x,u) \mathrm{d}\tau,
\end{equation}

where $T_{\text{sim}}$ is the simulation time set to reach steady system behavior.
The simulation study is done for $2601$ initial conditions ranging from $-5$ and $5$ in $0.2$ steps for $x_1$ and $x_2$.

Fig. \ref{fig:adp-nom} shows a cost comparison between the control strategies for the considered initial conditions using the said performance mark \eqref{eq:performance-mark}.
That means, the axis of Fig. \ref{fig:adp-nom} show the different intial conditions and the color bar indicates the cost difference between the control strategies based on the performance mark.
\vspace{-1em}
\begin{figure}[h!]
	\centering
	\subfigure{\includegraphics[width=0.24\textwidth]{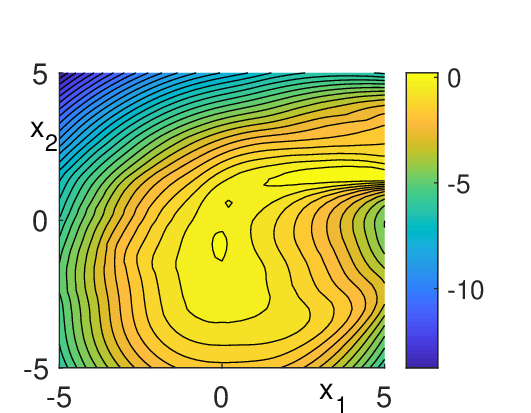}}
	\subfigure{\includegraphics[width=0.24\textwidth]{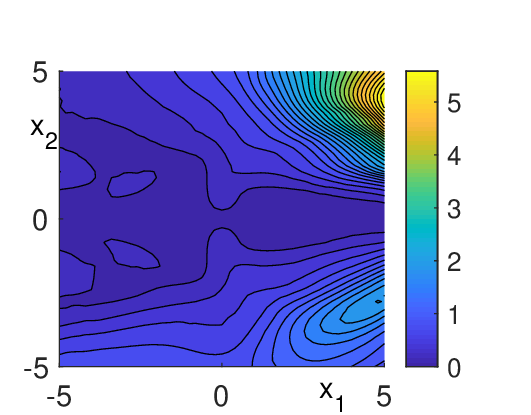}}
	\caption{Left: Difference of the performance mark \eqref{eq:performance-mark} between the proposed approach and the nominal controller. Right: Difference of the performance mark \eqref{eq:performance-mark} between the the proposed approach and the global optimal controller}
	\label{fig:adp-nom}
\end{figure}
For comparison reasons, the performance mark of the nominal controller for the initial state $x(0)=(-5 \; 5)^\top$ is $36$.
\vspace{-1em}
\begin{figure}[h!]
	\centering
	\includegraphics[scale=0.72]{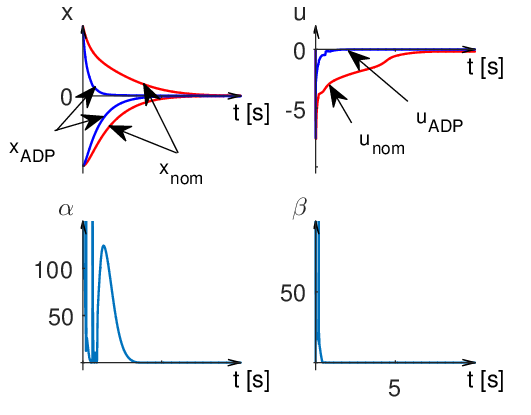}
	\caption{Trajectories of the state, input, $\alpha$ and $\beta$ for the intial state $x(0)=(-1 \; 1)^\top$}
	\label{fig:uti-nom}
\end{figure}
In Fig. \ref{fig:uti-nom}, trajectories of $x,u,\alpha$ and $\beta$ for a particular case of initial conditions are provided.
 
\section{Conclusion}
This paper introduced a particular ADP formulation with stability guarantees.
Based on a Lyapunov function constraint, the actor optimization problem in ADP was modified.
The calculation of the inverse Hessian for the optimization was avoided, by introducing a time-varying gradient descent method.
The resulting optimization problems for the critic and actor were analyzed with respect to convexity, feasibility, existence of solutions and convergence.
The convergence result in this paper is in the sense that convergence to a prescribed vicinity of the time-varying optima are achieved.
A simulation study for different initial conditions was provided, which shows cost improvement compared to a nominal controller.




\bibliography{bib/ADP-and-RL,bib/opt-ctrl,bib/extremum-seeking,bib/DP,bib/nonlin-ctrl,bib/stabilization,bib/dyn-sys,bib/opt-algo} 
\bibliographystyle{IEEEtran}

\end{document}

%% file: ACSD-preamble.tex
\usepackage{xspace}
\usepackage{soul}
\usepackage{mathtools}
\usepackage{makecell}
\usepackage{todonotes}
\usepackage{tikz}
\usepackage{nicefrac}

\newcommand{\ie}{\unskip, i.\,e.\xspace}
\newcommand{\eg}{\unskip, e.\,g.\xspace}

\newcommand{\R}{\ensuremath{\mathbb{R}}}

\newtheorem{theorem}{Theorem}
\newtheorem{lemma}{Lemma}
\newtheorem{assumption}{Assumption}
\newtheorem{remark}{Remark}
\newtheorem{definition}{Definition}
\DeclareMathOperator*{\argmin}{argmin}
